\documentclass[12pt]{article}
\usepackage[utf8]{inputenc}
\usepackage{amssymb,amsmath,amsfonts, eucal, mathrsfs} 
\usepackage[colorinlistoftodos]{todonotes}
\usepackage[allcolors=blue]{hyperref}
\newtheorem{theorem}{Theorem}

\newtheorem{lemma}[theorem]{Lemma}

\newtheorem{corollary}[theorem]{Corollary}

\newtheorem{remark}[theorem]{Remark}

\newcommand{\R}{\mathbb{R}}

\newcommand{\Ima}{\mbox{Im }}
\newcommand{\rank}{\mbox{rank }}

\newcommand{\po}{{\hspace*{-1ex}}{\bf .  }}

\newcommand{\End}{\mbox{End}}
\newcommand{\trace}{\mbox{tr\,}}

\def\<{{\langle}}
\def\>{{\rangle}}
\def\P{{\cal P}}

\def\bea{\begin{eqnarray*} }
\def\eea{\end{eqnarray*} }
\def\be{\begin{equation} }
\def\ee{\end{equation} }

\def\Ima{Im}
\newcommand{\Address}{{
  \bigskip
  \footnotesize

 \textsc{IMPA, Estrada Dona Castorina, 110, Rio de Janeiro, Brazil 22460-320}\par\nopagebreak
  \emph{E-mail address}: \texttt{mibieta@impa.br}
}}
\def\proof{\noindent\emph{Proof: }}
\def\qed{\ifhmode\unskip\nobreak\fi\ifmmode\ifinner
\else\hskip5 pt \fi\fi\hbox{\hskip5 pt \vrule width4 pt
height6 pt  depth1.5 pt \hskip 1pt }}

\begin{document}

\title{Infinitesimal bendings of complete Euclidean hypersurfaces}
\author{Miguel Ibieta Jimenez}
\date{}

\maketitle

\begin{abstract} A local description of the non-flat infinitesimally bendable Euclidean hypersurfaces was recently given by Dajczer and Vlachos \cite{DaVl}. From their classification,  it follows that there is an abundance of infinitesimally bendable hypersurfaces
that are not isometrically bendable.
In this paper we consider the case  of complete hypersurfaces $f\colon M^n\to\R^{n+1}$, $n\geq 4$.
If there is no open subset where $f$ is either totally geodesic or a cylinder over an unbounded hypersurface of $\R^4$,  we prove that $f$ is infinitesimally bendable only along ruled strips. In particular, if the hypersurface is simply connected, this implies that any infinitesimal bending of $f$ is the variational field of an isometric bending. 
\end{abstract}

Given an isometric immersion of a Riemannian manifold $M^n$, $n\geq3$, into Euclidean space $\R^{n+1}$
a natural question is whether there exist, locally or globally, other isometric immersions, apart from compositions with rigid motions of $\R^{n+1}$. 
If such immersions exist they are called isometric deformations. In the absence of isometric deformations the submanifold is said to be \emph{rigid}.

In the  local case, the classical Beez-Killing theorem states that a hypersurface having at least three nonzero principal curvatures at any point is rigid. 
On the other hand, any hypersurface whose second fundamental form has rank at most one is flat, and hence highly isometrically deformable. 
The remaining case is when the second fundamental form has constant rank two.  In that case, 
the ones that admit isometric deformations were parametrically described
by Sbrana \cite{Sb1} and Cartan \cite{Ca1}.  
A modern presentation of these alternative descriptions, as well as further results, can be found in \cite{DaFlTo}.

An \emph{isometric bending} of an isometric immersion $f\colon M^n\to\R^{n+1}$ is a smooth map 
$$
F\colon(-\epsilon,\epsilon)\times M^n\to \R^{n+1}
$$ 
such that $F(0,\cdot)=f(\cdot)$ and $F(t,\cdot)$ is an isometric immersion for each $t\in(-\epsilon,\epsilon)$. 

A way to construct an isometric bending is to compose $f$ with a smooth one-parameter family of rigid motions in $\R^{n+1}$, that is, 
\begin{equation*}
F(t,x)=C(t)f(x)+v(t)
\end{equation*}
where $C(t)$ is an orthogonal transformation of $\R^{n+1}$ and $v(t)\in\R^{n+1}$ for each $t\in(-\epsilon,\epsilon)$. 

An isometric bending of $f$ given by the expression above is said to be \emph{trivial}. If $f$ admits a non-trivial isometric bending then it is called \emph{isometrically bendable}. Otherwise, $f$ is said to be  \emph{isometrically unbendable}.

In the global case, Sacksteder \cite{Sa} proved that any compact Euclidean hypersurface in $\R^{n+1}$ is isometrically unbendable. 
Dajczer and Gromoll \cite{DaGr} showed that a complete hypersurface in $\R^{n+1}$ is isometrically bendable only along ruled strips, 
provided that it does not contain an open subset that is a cylinder over an unbounded hypersurface in $\R^{4}$.
Recall that
a hypersurface is said to be \emph{ruled} if it admits a foliation of codimension one whose leaves (rulings) are part of affine subspaces of $\R^{n+1}$.  For a hypersurface with boundary we say it is ruled if, in addition to the previous condition, 
the rulings are tangent to the boundary. When the rulings are all complete we say the hypersurface is completely ruled and any of its connected components is called a \emph{ruled strip}.
Dajczer and Gromoll \cite{DaGr} proved that any simply connected ruled strip without flat points, that does not contain a cylinder over a surface in $\R^3$, admits only ruled isometric deformations with the same rulings. 
Moreover, these deformations are in one to one correspondence with the smooth functions on an open interval. 

A weaker notion of bending of $f\colon M^n\to\R^{n+1}$ is associated to variations of $f$ by hypersurfaces whose induced metrics are equal up to the first order. A precise definition of this notion is given next.

Let $\mathfrak{X}(M)$ denote the set of tangent vector fields of $M^{n}$ and $\Gamma(E)$ the sections of a bundle $E$ over $M^n$. Let $F$ be an isometric bending of $f$ and let $\tau(x)=F_{*}\partial / \partial t|_{t=0}(x)\in\Gamma(f^{*}T\R^{n+1})$ be the associated variational field.
Then $\tau$ satisfies
\begin{equation}\label{infben}
\<\tilde{\nabla}_{X}\tau,f_{*}Y\>+\<\tilde{\nabla}_{Y}\tau,f_{*}X\>=0
\end{equation}
for any $X,Y\in \mathfrak{X}(M),$ where $\tilde{\nabla}$ denotes the connection of the Euclidean space. 

A vector field $\tau\in\Gamma(f^{*}T\R^{n+1})$ satisfying \eqref{infben} is called an \emph{infinitesimal bending} of the immersion $f$. To any infinitesimal bending $\tau$ of $f$ we always associate the variation 
\begin{equation}\label{variation}
F(t,x)=f_t(x)=f(x)+t\tau(x),\;\;t\in\R,
\end{equation}
that has $\tau$ as variational vector field. Then, the maps $f_t$ are immersions since 
\begin{equation}\label{metrica t}
\< f_{t*}X,f_{t*}Y\>=\< f_{*}X,f_{*}Y\> +t^{2}\< \tilde{\nabla}_{X}\tau,\tilde{\nabla}_{Y}\tau\>
\end{equation}
for all $X,Y\in\mathfrak{X}(M)$ and $t\in\R$.

Clearly any isometric bending determines an infinitesimal one. In this way, infinitesimal bendings given by trivial isometric bendings are said to be \emph{trivial} infinitesimal bendings. More precisely, a trivial infinitesimal bending $\tau$ has the form
$$
\tau(x)={\cal D}f(x)+w
$$
where ${\cal D}$ is a skew-symmetric linear endomorphism of $\R^{n+1}$ and $w\in\R^{n+1}$. In the sequel, we always identify two infinitesimal bendings $\tau_{1}$ and $\tau_{2}$ if $\tau_{1}-\tau_{2}$ is trivial.

The problem of existence of non-trivial infinitesimal bendings of an isometric immersion, gives rise to another notion of rigidity.   
Namely, an isometric immersion is said to be \emph{infinitesimally rigid} if it only admits trivial infinitesimal bendings. Otherwise, we say that the submanifold  is \emph{infinitesimally bendable}. 

For simplicity, from now on  we refer to the rank of the  second fundamental form of an isometric immersion $f\colon M^n\to\R^{n+1}$ as the rank of the immersion $f$. 
It is a classical result that hypersurfaces of rank at least three at every point are infinitesimally rigid; see \cite{DaRo}. 
The case of constant rank two was studied by Dajczer and Vlachos \cite{DaVl}. They completed the work of Sbrana \cite{Sb2} describing locally all hypersurfaces admitting non-trivial infinitesimal bendings. It turns out that the class of infinitesimally bendable hypersurfaces is much larger than that of isometrically bendable ones. 
From \cite{DaVl} we also have that infinitesimal bendings of cylinders of rank 2 are determined by infinitesimal bendings of the non-flat factor. 
 
\medskip
In this paper, we consider the class of complete Euclidean hypersurfaces that carry a non-trivial infinitesimal bending.
Using facts from \cite{DaVl} and inspired by Theorem $3.4$ in \cite{DaGr}, we prove the following result.

\begin{theorem}\po\label{main}
Let $f\colon M^n\to\R^{n+1}$, $n\geq 4$, be an isometric immersion of a complete Riemannian manifold. Assume that there is no open subset of $M^n$ where $f$ is either totally geodesic or a cylinder over a hypersurface in $\R^{4}$ with complete one-dimensional leaves of relative nullity. Then $f$ admits non-trivial infinitesimal bendings only along ruled strips.
\end{theorem}

The precise definition of the relative nullity foliation in the above statement is given in the next section.

Notice that any vector field normal to a hypersurface along a totally geodesic subset is trivially an infinitesimal bending. In particular, a compact Euclidean hypersurface is infinitesimally rigid provided that it has no open totally geodesic subsets \cite{DaRo}.
Observe also that the hypothesis of the theorem also rules out open subsets that are cylinders over curves in $\R^2$ as well as cylinders over surfaces in $\R^3$. 

We point out that the existence of complete nonruled isometrically bendable hypersurfaces of constant rank two in $\R^4$, that are not surface-like, is an open problem \cite{DaGr}. By surface-like we mean a hypersurface that is a cylinder over either a surface in $\R^3$ or a cone over a surface in an umbilical submanifold of $\R^4$.

It was shown in \cite{DaVl} that any infinitesimal bending of a simply connected (non-flat) ruled hypersurface, that is not surface-like on any open subset, is the variational vector field of an isometric bending. Thus, any infinitesimal bending of a simply connected complete hypersurface in $\R^{n+1}$ satisfying the assumptions of Theorem~\ref{main}, is the variational vector field of an isometric bending, in contrast with what happens in the local case; see Corollary~\ref{corollary}.

The author would like to thank M. Dajczer for his valuable comments on preliminary versions of this work.
\section{Preliminaries}

In this section, we present several results to be used for the proof 
of Theorem~\ref{main} given in the following section.

\vspace{1ex}

Given an isometric immersion $f\colon M^n\to \R^{n+1}$ of a Riemannian manifold $M^n$, $n\geq 4$, let $\tau\in\Gamma(f^{*}T\R^{n+1})$ be an infinitesimal bending of $f$ and let $f_t$, $t\in\R$,  be the immersions given by \eqref{variation}.

\begin{lemma}\po\label{rango3} The immersions $f_t$ and $f_{-t}$ induce the same metric on $M^n$ for any $t\in\R$. Moreover, if $f$ is not totally geodesic and there is $0\neq t_0\in\R$ such that $f_{t_0}$ and $f_{-t_0}$ are congruent then $\tau$ is trivial.
\end{lemma}
\begin{proof}
See Theorem 1 in \cite{DaRo}. \vspace{2ex}\qed
\end{proof}

If $g_t$ denotes the metric on $M^n$ induced by $f_t$, we have from \eqref{metrica t} that 
$$
\partial/\partial t|_{t=0} g_t(X,Y)=0
$$ 
for all $X,Y\in \mathfrak{X}(M)$. Thus, the corresponding Levi-Civita connections 
satisfy
\begin{equation} \label{derivada conex}
\partial/\partial t|_{t=0} \nabla_{X}^tY=0 
\end{equation} 
for all $X,Y\in \mathfrak{X}(M)$.

Let $N(t)$ be a smooth one-parameter family of unitary vector fields, such that $N(t)$ is normal to the hypersurface $f_t$. Then let $A(t)$ stand for the corresponding second fundamental forms. In the sequel, we write $N=N(0)$ and $A=A(0)$. Let $\xi\in\Gamma(f^{*}T\R^{n+1})$ and $B\in\Gamma(\End(TM))$ be given by  
\begin{equation}\label{def B}
\xi=\partial/\partial t|_{t=0}N(t)\;\;\mbox{and}\;\; B=\partial/\partial t|_{t=0}A(t)
\end{equation}
respectively. 
In addition, let $L\in \Gamma(\End(TM,f^{*}T\R^{n+1}))$ be defined by 
$$
LX=\tilde{\nabla}_{X}\tau
$$ 
where $\tilde{\nabla}$ is the connection in $\R^{n+1}$.
\begin{lemma}\po
The tensor $L$ and the vector field $\xi$ satisfy
\begin{equation}\label{derivada del normal}
\<\xi,N\>=0
\end{equation}
and 
\begin{equation}\label{relacion l con el normal}
\<\xi, f_{*}X\>+\< N, LX\>=0
\end{equation}
for all $X\in \mathfrak{X}(M)$.
\end{lemma}
\begin{proof}
The derivative of $\< N(t), N(t)\>=1$ with respect to $t$  at $t=0$ yields \eqref{derivada del normal}. In a similar way $\< N(t),f_{t*}X\>=0$ provides \eqref{relacion l con el normal}. \qed
\end{proof}

\begin{lemma}\po
The covariant derivatives of $L$ and $\xi$ are given by
\begin{equation} \label{derivada de L}
(\tilde\nabla_{X}L)Y=\< BX,Y\> N+\< AX,Y\> \xi
\end{equation}
and \begin{equation}\label{derivada de xi}
\tilde{\nabla}_{X}\xi=-f_{*}BX-LAX
\end{equation}
for all $X,Y\in \mathfrak{X}(M)$.
\end{lemma}
\begin{proof}
See Lemma 9 in \cite{DaVl}. \vspace{2ex}\qed
\end{proof}

In addition, we have the following result.   
\begin{lemma}\po\label{B es Codazzi}
The tensor $B$ is symmetric and satisfies
\begin{equation}\label{B1}
BX\wedge AY-BY\wedge AX=0
\end{equation}
and
\begin{equation}\label{B2}
(\nabla_{X}B)Y=(\nabla_{Y}B)X
\end{equation} 
for all $X,Y\in \mathfrak{X}(M)$. 
\end{lemma}
\proof  See Proposition $10$ in \cite{DaVl}.\vspace{2ex}\qed

If $\tau$ is a trivial infinitesimal bending, it follows from \eqref{derivada de L} that 
$$
\<BX,Y\>=\<(\tilde{\nabla}_X L)Y,N\>=\<{\cal D}(\tilde{\nabla}_X f_*)Y,N\>
=\<AX,Y\>\<{\cal D}N,N\>=0
$$
for all $X,Y\in\mathfrak{X}(M)$. Hence, the symmetric tensor  associated to a trivial infinitesimal bending is $B=0$. 
Moreover,  we have from \cite{DaVl} that if the symmetric tensor $B$ associated to an infinitesimal bending $\tau$ is identically zero then $\tau$ is trivial.
\vspace{2ex}

The \emph{relative nullity subspace} $\Delta(x)\subset T_{x}M$ of $f$ at $x\in M^n$ is 
$\Delta(x)=\ker A(x)$. The dimension $\nu(x)$ of $\Delta(x)$ is called the \emph{index of relative nullity}.
\vspace{1ex}

It is a standard fact that on an open subset $U\subset M^n$ where $\nu(x)$ is constant, the \emph{relative nullity distribution} 
$x\in U\mapsto\Delta(x)$ is totally geodesic and the leaves of the  \emph{relative nullity foliation} are totally geodesic submanifolds of $\R^{n+1}$.\vspace{1ex}

We recall that the \emph{splitting tensor} $C\colon\Gamma(\Delta)\to\Gamma(\End( \Delta^{\perp}))$ is defined by
$$
C_{T}X=-(\nabla_{X}T)_{\Delta^{\perp}},
$$
where $T\in \Gamma(\Delta)$ and $X\in \Gamma(\Delta^{\perp})$; see \cite{DaGr}. It follows from the Codazzi equation that the splitting tensor satisfies 
\begin{equation}\label{A1}
\nabla_{T}A=AC_{T}=C_{T}'A
\end{equation}
where $C'_{T}$ denotes the transpose of $C_{T}$. 
\vspace{2ex}

The following is a fundamental result in the theory of isometric immersions.
\begin{lemma}\po\label{lema eq}Let $f\colon M^n\to \R^{n+1}$ be an isometric immersion and let $U\subset M^n$ be an open subset where  $\nu(x)=\nu_0$ is constant. 
Let $\gamma\colon\left[0,a\right]\to M^n$ be a geodesic such that $\gamma(\left[0,a\right))\subset U$ is contained in a leaf of the relative nullity foliation. Then $\nu(\gamma (a))=\nu_0$ and the splitting tensor $C_{\gamma'}$ has a smooth extension $\bar{C}_{\gamma'}$ to $\left[0,a\right]$ that satisfies the differential equation 
\begin{equation*}
\nabla_{\gamma'(s)}A=A\bar{C}_{\gamma'(s)}.
\end{equation*}
\end{lemma}
\begin{proof}
See Theorem 5.3 and Lemma 6.15 in \cite{Da}.\vspace{2ex}\qed
\end{proof}

In the following result $\nabla^h$ denotes the connection induced on $\Delta^\perp$.  
\begin{lemma}\po\label{equation CTXY}
We have that 
\begin{equation*}
(\nabla^h_{X}C_{T})Y-(\nabla^h_{Y}C_{T})X=C_{(\nabla_XT)_{\Delta}}Y-C_{(\nabla_{Y}T)_{\Delta}}X
\end{equation*} 
for any $X,Y\in\Gamma(\Delta^{\perp})$ and $T\in\Gamma(\Delta)$.
\end{lemma}
\begin{proof}
See Lemma 1.5 in \cite{DaGr}.\vspace{2ex}\qed
\end{proof} 

When the leaves of the relative nullity are complete we have the following. 
\begin{lemma}\po\label{lema autovalor}Let $f\colon M^n\to \R^{n+1}$ be an isometric immersion.
Assume that $U\subset M^n$ is an open subset where $\nu(x)=\nu_0$ is constant and the relative nullity leaves are complete. Then, for any $x_0\in U$ and $T_0\in\Delta(x_0)$ the only possible real eigenvalue of $C_{T_0}$ is zero. Moreover, if $\gamma(s)$ is a geodesic through $x_0$ tangent to $T_0$ then
\be\label{eq:dif}
C_{\gamma'(s)}=\P_0^s C_{T_0}(Id-sC_{T_0})^{-1}(\P_0^s)^{-1}
\ee
where $\P_0^s$ is the parallel transport along $\gamma$ from $x_0$. In particular, $\ker C_{\gamma'}$ is parallel along the geodesic $\gamma$.
\end{lemma}
\begin{proof}
See Lemma 1.8 in \cite{DaGr}.\vspace{2ex}\qed
\end{proof}

An isometric immersion $G\colon M^{n-k}\times \R^{k}\to \R^{m}$ of the Riemannian product $M^{n-k}\times \R^{k}$ is called a \emph{cylinder} over the isometric immersion $g\colon M^{n-k}\to\R^{m-k}$, if it factors as $$G=g\times Id\colon M^{n-k}\times \R^{k}\to \R^{m-k}\times\R^{k}$$ where $Id\colon\R^{k}\to\R^{k}$ is the identity.\vspace{2ex}

Notice that, on open sets where $\nu (x)$ is constant, the distribution $\Delta^{\perp}$ is totally geodesic if and only if $C$ is identically zero. In such case we have that the manifold is locally a Riemannian product and the immersion is locally a piece of a cylinder.

The following result is useful when characterizing hypersurfaces of constant rank $2$ whose relative nullity leaves are complete.  

\begin{lemma}\po\label{cokerC}
Let $f\colon M^n\to \R^{n+1}$ be an isometric immersion. If $U\subset M^n$ is an open subset where $f$ has constant rank $2$ and the leaves of the relative nullity are complete, then the codimension of 
$$
C_0(x)=\{T\in\Delta(x) : x\in U \;\;\mbox{and}\;\;C_{T}=0\}
$$
is at most one. Moreover, if $\dim C_0^{\perp}=1$ and $C_{T}$ is invertible
for $T\in\Gamma(C_0^{\perp})$, then $f|_U$ is a cylinder over a hypersurface $g\colon L^{3}\to \R^{4}$ that carries a one-dimensional relative nullity distribution with complete leaves. 
\end{lemma}

\begin{proof}
See Lemmas 1.9 and 1.10 in \cite{DaGr}.
\vspace{2ex}\qed
\end{proof} 

If $f\colon M^n\to \R^{n+1}$ has rank $2$, we have from \eqref{B1} that $\rank B\leq 2$ and $\Delta\subset\ker B$. Then \eqref{B2} implies that
\begin{equation}\label{equation for B and C}
\nabla_T B=BC_T=C_T' B
\end{equation}
for any $T\in\Gamma(\Delta)$. 

\begin{lemma}\po\label{equation for B}
Let $f\colon M^n\to \R^{n+1}$ be an isometric immersion and let $U\subset M^n$ be an open subset where $f$ has rank $2$. Let
$\tau$ be an infinitesimal bending not trivial on any open subset of $U$. If
$\gamma\colon [0,a]\to M^n$ is a unit speed geodesic such that $\gamma([0,a))\subset U$ is contained in a leaf of the relative nullity foliation, then $B$ satisfies 
\begin{equation}\label{equation B extended}
\nabla_{\gamma'(s)}B=B\bar{C}_{\gamma'(s)}
\end{equation} on $[0,a]$ where $\bar{C}_{\gamma'(s)}$ was given by Lemma~\ref{lema eq}.
\end{lemma}

\begin{proof} We claim that the relative nullity distribution $\Delta_t$ of the immersion $f_t$ in \eqref{variation} satisfies $\Delta_t=\Delta_0=\Delta$ along $\gamma$ in $U$.
To see this, it suffices to argue for small values of $t$. Thus, the second fundamental form 
$A(t)$ of $f_t$ has rank at least $2$ in a neighborhood of $\gamma([0,a])$. If the rank of $A(t)$ is larger than $2$ on an open subset, then 
the Beez-Killing rigidity theorem and Lemma~\ref{rango3} yield that $\tau$ is trivial on that open subset, and that has been excluded in $U$.   
Therefore, for $t$ small enough we conclude that $A(t)$ has constant rank $2$ on a neighborhood of $\gamma([0,a))$ in $U$. 

The normal vector field $N(t)$ of $f_t$ decomposes as 
$$
N(t)=Z(t)+b N
$$ 
where $Z(t)\in f_{*}TM$ and $b=b(t,x)=\<N(t),N\>$.
At any $x\in M^n$, we easily obtain using \eqref{infben} and \eqref{relacion l con el normal} that
$$
0=\< N(t), f_{t*}X\>=\< (Id-tL_0)Z(t)-tb\xi,f_{*}X\>
$$
for any $X\in\mathfrak{X}(M)$, where $L_0=\pi\circ L$ and $\pi\colon f^{*}T\R^{n+1}\to f_{*}TM$ is the orthogonal 
projection. From \eqref{derivada del normal} we have 
\begin{equation}\label{ecu Z}
(Id-tL_0)Z(t)=tb\xi,
\end{equation}
and thus
\begin{equation}\label{ecu Z 2}
Z(t)=tbS(t)\xi
\end{equation}
where  $S(t)=\left(Id-tL_0\right)^{-1}$. 

We know that $\Delta\subset\ker B$ on $U$. Since $\pi$ is parallel along the relative nullity leaves, 
we obtain from \eqref{derivada de L}  that 
$$
\tilde{\nabla}_T L_0=\tilde{\nabla}_T(\pi\circ L)=0
$$
and from \eqref{derivada de xi} that
$$
\tilde{\nabla}_T\xi=0
$$
for any $T\in\Gamma(\Delta)$ on $U$. Hence, taking the covariant derivative of \eqref{ecu  Z} 
in the direction of $T\in\Gamma(\Delta)$ we have 
$$
\tilde{\nabla}_{T}Z(t)=tT(b)S(t)\xi.
$$
We obtain from \eqref{ecu Z 2} that 
$$
\tilde{\nabla}_{T}N(t)=T(b)\left(tS(t)\xi+N\right)=\frac{T(b)}{b}N(t).
$$
Since $N(t)$ is a unitary vector field, then $\tilde{\nabla}_{T}N(t)=0$
and the claim follows.

It follows from \eqref{derivada de L} that $L$ is parallel along the relative nullity leaves on $U$. Notice that $f\circ\gamma$ describes a line segment. Then, we have
$$
\frac{\tilde{D}}{ds} (f_{t*}\gamma')=\frac{\tilde{D}}{ds}(f_*\gamma'+tL\gamma')=0
$$
where $\tilde{D}/ds$ is the covariant derivative along $f_t\circ\gamma$ in $\R^{n+1}$. We conclude that $f_{t}\circ\gamma$ is a
geodesic in $\R^{n+1}$. 

From Lemma~\ref{lema eq} we see that
\begin{equation}\label{At equation}
\nabla^t_{\gamma'(s)}A(t)=A(t)\bar{C}_{\gamma'(s)}^t
\end{equation}
for any $s\in[0,a]$. Let $\pi_t$ be the orthogonal projection $\pi_t\colon TM\to \Delta_t^{\perp}$ with respect to the metric induced by $f_t$. 
Extend ${C}_{\gamma'(s)}^t$ to $T_{\gamma(s)}M$ as
$$
C_{\gamma'(s)}^tX=-\pi_t\nabla^t_{X}\gamma'(s)
$$
for $s\in[0,a)$. Define $\bar{C}_{\gamma'(a)}^t T=0$ for $T\in\Delta(a)$.
Fix $x=\gamma(s)\in U$ and $X\in T_xM$. Since the immersions $f_t$ and $f_{-t}$ induce the same metric we have that $\pi_{t}=\pi_{-t}$ on $T_xM$. 
Therefore $\pi_tX$ is an even function of $t$ from an interval $(-\epsilon,\epsilon)$ into $T_xM$. Hence its derivative at $t=0$ vanishes.
From this and \eqref{derivada conex}, we obtain that the linear operators $C_{\gamma'(s)}^t$ on $T_xM$ satisfy
\begin{align*}
\partial/\partial t|_{t=0}C_{\gamma'(s)}^tX&=-\partial/\partial t|_{t=0}(\pi_t\nabla_{X}^t\gamma'(s))\\
&=-\partial/\partial t|_{t=0}\pi_t\nabla_X\gamma'(s)-\pi_0\partial/\partial t|_{t=0}\nabla_X^t\gamma'(s)\\
&=0
\end{align*}
for any $s\in[0,a)$. It follows that $\partial/\partial t|_{t=0}\bar{C}_{\gamma'(a)}^t=0$. 
Hence, taking the derivative of \eqref{At equation} with respect to $t$ at $t=0$ gives 
\eqref{equation B extended}.\vspace{2ex}\qed
\end{proof}

Next we discuss several facts about ruled Euclidean hypersurfaces. 
\vspace{1ex}

Let $f\colon M^n\to \R^{n+1}$ be a ruled hypersurface and let
$c\colon I\to M^n$ be a unit speed curve orthogonal to the rulings. The rulings form an affine vector bundle 
over $\tilde c=f\circ c$ in $\R^{n+1}$. Then let $T_i(s)$, $1\leq i\leq n-1$, be orthonormal tangent fields on the corresponding bundle along $c$ which are parallel with respect to the induced connection. Set $f_*c'=\tilde{T}_0$, $\tilde{T}_i=f_*T_i$ and 
let $N$ be  a unit vector field along $c$ normal to $f$. We have
\begin{equation*}
\begin{cases}
\tilde{\nabla}_{\partial/\partial s}\tilde{T_0}=-\sum_{i=1}^{n-1}\varphi_i \tilde{T}_i+\theta N
\vspace{1ex}\\
\tilde{\nabla}_{\partial/\partial s}\tilde{T_i}=\varphi_i\tilde{T}_0+\beta_iN
\end{cases}
\end{equation*}
where $\theta=\<AT_0,T_0\>$, $\varphi_i=\<\nabla_{T_0}T_i,T_0\>$ and $\beta_i=\<AT_i,T_0\>$.

We parametrize a neighborhood of $\tilde{c}$ in $f(M)$ by means of $\tilde{f}\colon W\subset I\times\R^{n-1}\to\R^{n+1}$ given by
\begin{equation} \label{parametrization ruled}
\tilde{f}(s,u_1,\ldots,u_{n-1})=\tilde{c}(s)+\sum_{i=1}^{n-1} u_i\tilde{T}_i(s).
\end{equation}
We have at $(s,u_1,\ldots,u_{n-1})$ that
$$
\tilde{f}_*\partial/\partial s=(1+\sum_iu_i\varphi_i)\tilde{T}_0+\sum_iu_i\beta_iN.
$$
Therefore, the map $\tilde{f}$ has maximal rank if and only if 
$$
|\tilde{f}_*\partial/\partial s|^2=(1+\sum_iu_i\varphi_i)^2+(\sum_iu_i\beta_i)^2\neq 0.
$$
Note that the directions for which $\sum_iu_i\beta_i=0$ are in the relative nullity of $f$ at $c(s)$.
\medskip

The following result can be found in \cite{DaGr} for submanifolds of arbitrary codimension. 
But for the convenience of the reader we give a proof in the hypersurface case.

\begin{lemma}\po \label{regladas 2}
Let $f\colon M^n\to\R^{n+1}$ be an isometric immersion and let $U\subset M^n$ be an open subset where $f$ has constant rank $2$. Assume that $f|_U$ is ruled and has complete relative nullity leaves. Suppose that $\delta\colon [0,a]\to M^n$ is a unit speed geodesic orthogonal to $\Delta$ such that $\delta[0,a)\subset U$ is contained on a ruling. Then the rank of $f$ at $\delta(a)$ is $2$. 
Moreover, every point in $U$ has a neighborhood $V$ such that $f|_V$ extends to a ruled strip of constant rank 2. 
\end{lemma}

\begin{proof} Let $W\subset I\times \R^{n-1}$ be an open subset where the parametrization \eqref{parametrization ruled} is defined and 
write $W_s=W\cap \{s\}\times \R^{n-1}$. Assume that the geodesic $\delta$ is contained on the ruling determined by $\tilde{f}|_{W_{s}}$ and has $T_{n-1}$ as its tangent vector field. Notice that $\tilde{f}(s,0,\ldots,0,r)$ is a parametrization of $\delta$. Since $\beta_{n-1}(s)\neq 0$ the map $\tilde{f}$ has maximal rank along $\delta$ and at $\delta(r)$ we have 
$$
\tilde{f}_* (\partial/\partial s)= (1+r\varphi_{n-1})\tilde{T}_0+r\beta_{n-1}N.
$$
Let $\tilde{N}(\delta(r))=\alpha(r)\tilde{T}_0+N$ be a vector field normal to $f$ along $\delta$ (not necessarily unitary). Then
$$
0=\<\tilde{f}_* (\partial/\partial s),\alpha\tilde{T}_0+N\>=\alpha(1+r\varphi_{n-1})+r\beta_{n-1}.
$$
Taking $r\in (0,a]$ we see that $1+r\varphi_{n-1}\neq 0$, then
$$
\alpha(r)=-\frac{r\beta_{n-1}}{(1+r\varphi_{n-1})}\cdot
$$
Therefore, we have that
\begin{align*}
\<\tilde{\nabla}_{\delta'(r)}\tilde{f}_* (\partial/\partial s), \tilde{N}(\delta(r))\>&=\<\varphi_{n-1}\tilde{T}_0+\beta_{n-1}N, \alpha(r)\tilde{T}_{0}+N\>\\
&=\beta_{n-1}\left(\frac{1}{1+r\varphi_{n-1}}\right)
\end{align*}
which does not vanish.
Thus the rank of $\tilde{f}$ at $\delta(a)$ is $2$ and hence the same holds for $f$. 

It remains to prove that $f|_U$ extends locally to a ruled strip. Fix $x\in U$ and let $V\subset U$ be a neighborhood of $x$ parametrized by \eqref{parametrization ruled}. Extend $\tilde{f}$ to $I\times\R^{n-1}$ with the same expression. We claim that this extension defines a ruled strip of constant rank 2. 
We first prove that $\tilde{f}$ has no singular points. As seen previously, $\tilde{f}$ is singular at points where
$$
(1+\sum_iu_i\varphi_i)^2+(\sum_iu_i\beta_i)^2=0.
$$
Then, it suffices to prove that $\sum_{i}u_i\varphi_i=0$ for any $T=\sum_iu_iT_i(s)\in \Delta(c(s))$.

Given $T\in\Delta(c(s))$, we have that
$$
\sum_iu_i\varphi_i=\<\nabla_{T_0}T,T_0\>=-\<C_TT_0,T_0\>.
$$
If the splitting tensor vanishes there is nothing to prove. 
Otherwise, if $X$ is a unit vector field on $V$ tangent to a ruling and orthogonal to the relative nullity, it follows from the completeness of the relative nullity leaves of $f$ that $C_TX=0$ for any $T\in\Gamma(\Delta)$.
Finally, since the only real eigenvalue of $C_T$ is zero by  Lemma~\ref{lema autovalor}, then
$\<C_TT_0,T_0\>=0$,
and thus $\tilde{f}$ has no singular points.

It follows from Lemma~\ref{lema eq} that the open subset where $\tilde{f}$ has rank two is a union of complete relative nullity leaves. From the previous discussion we have that the rank of $\tilde{f}$ along any ruling  is two, and the claim follows.\qed
\end{proof}

\begin{lemma}\po\label{B para reglada}
Let $f\colon M^n\to \R^{n+1}$ be a ruled hypersurface of constant rank $2$ with complete relative nullity leaves. 
Assume that the splitting tensor $C$ does not vanish on any open subset. If $\tau$ is an infinitesimal bending of $f$, then its associated symmetric tensor $B$ satisfies 
\begin{equation}\label{forma de B reglada}
B|_{\Delta^\perp}=\begin{bmatrix}\theta & 0\\ 0 & 0
  \end{bmatrix}
\end{equation}
with respect to a local orthonormal basis $\{Y,X\}$ of $\Delta^{\perp}$ such that $Y$ is orthogonal to the rulings. Moreover, the smooth function $\theta$ verifies 
\begin{equation}\label{ecu reglada}
X(\theta)=\< \nabla_{Y}Y,X \>\theta.
\end{equation} 
\end{lemma}

\begin{proof}
On the open dense subset where $C\neq 0$, let $T\in\Gamma(C_0^{\perp})$ be unitary. Locally take $X,Y\in \Gamma(\Delta^{\perp})$ orthonormal such that $Y$ is orthogonal to the rulings. 
We have seen that $X\in\Gamma(\ker C_T$). Moreover, Lemma~\ref{lema autovalor} implies that $C_T=\mu J$ for some smooth function $\mu$, where $J\in\Gamma(\End(\Delta^{\perp}))$ is defined by $JX=0$ and $JY=X$. 

We denote the restrictions of $A$ and $B$ to $\Delta^\perp$ by the same letters and let $D\in\Gamma(\End(\Delta^{\perp}))$ be given by 
$D=A^{-1}B.$ 
From \eqref{A1} and \eqref{equation for B and C} we have 
$$
ADC_T=C_T'AD=AC_TD.
$$
Hence $A[D,C_T]=0$, and thus $D$ commutes with $J$. This gives
$D=\phi_1 Id+\phi_2J$
and 
$$
B=\phi_1 A+\phi_2 AJ.
$$
Since the immersion is ruled, then $A$ has the form
$$
A=\begin{bmatrix}\lambda & \nu\\ \nu & 0
  \end{bmatrix}.
$$
We easily have from \eqref{B1} that $\phi_1=0$, and therefore $B$ has the form \eqref{forma de B reglada}.
Finally \eqref{ecu reglada} follows from \eqref{B2}.\qed
\end{proof}

\begin{remark}\po{\em From the above and Theorem 13 in \cite{DaVl}, the set of infinitesimal bendings of a ruled hypersurface satisfying the assumptions of the preceding result, is in one to one correspondence with the set of smooth functions on an interval.}
\end{remark}

\section{The proof}

We are now in condition to give the proof of Theorem~\ref{main}.
\vspace{1ex}

Let $\tau$ be a non-trivial infinitesimal bending of $f$ and let $B$ be the associated symmetric 
tensor given by \eqref{def B}.  We consider the subsets of $M^n$ defined by
$$
M_{i}=\{x\in M^n : \rank A(x)\geq i\}.
$$ 
Then $B|_{M_{3}}=0$ since $f$ is infinitesimally rigid on $M_{3}$.
Let $V\subset W_{2}=M_{2}\setminus \bar{M_{3}}$ be the open subset of $M^n$ defined by 
$$
V=\left\{x\in W_{2} : B(x)\neq 0\right\}.
$$ 

We claim that the leaves of relative nullity in $V$ are complete. Otherwise,
there is a geodesic $\gamma\colon [0,a]\to M^n$ contained in a leaf of the relative nullity such that 
$\gamma(\left[0,a\right))\subset V$ and $\gamma(a)\notin V$. 
From Lemma~\ref{equation for B} we have that $B$ satisfies \eqref{equation B extended} on $\left[0,a\right]$ with $B(a)=0$. Moreover, since $\nabla_{\gamma'(s)}B$ is a symmetric tensor we have that
\begin{equation}\label{ecu B extendida 2}
\nabla_{\gamma'(s)}B=\bar{C}_{\gamma'(s)}'B
\end{equation}
on $[0,a]$.
Take a parallel orthonormal basis of $\Delta^{\perp}$ along $\gamma$ and regard \eqref{ecu B extendida 2} as a differential equation of matrices. 
It follows from Jacobi's formula that
$$
\frac{d}{d s}(\det B(s))=\trace\left(\mbox{adj}(B)\,\frac{d B}{d s}\right).
$$
Thus
\begin{equation*}
\det B(s)=e^{\int_{0}^s\trace \bar{C}_{\gamma'(r)}'dr}\det B(0)
\end{equation*}
on $[0,a]$. If $B(0)$ has rank $2$ we already have a contradiction.
If $B(0)$ has rank $1$, we may assume that its first column is not zero and replace the remaining column so
the resulting matrix $\tilde B(0)$ has rank $2$. Taking $\tilde{B}(0)$ as a new initial condition for \eqref{ecu B extendida 2}, the corresponding solution $\tilde{B}(s)$ has maximal rank for any $s\in[0,a]$.
Notice that each column $w(s)$ of $B(s)$ verifies
\begin{equation*}
\frac{d}{ds}w(s)=\bar{C}_{\gamma'(s)}'w(s).
\end{equation*}
Since the first columns of $\tilde{B}(0)$ and $B(0)$ coincide the same holds for the first columns of $\tilde{B}(s)$ and $B(s)$. This again leads to a contradiction, and proves the claim.

We show next that $f|_V$ is ruled using arguments from the proof of Proposition 2.1 in \cite{DaGr}. By Lemma~\ref{cokerC} the codimension of $C_0$ is at most one. The assumption that $f(M)$ does not contain a cylinder gives that the subset 
$$
V_0=\{x\in V : C(x)=0\}
$$ 
has empty interior. 
Let $T\in\Gamma(\Delta)$ be a local unit vector field on the open subset $V_1=V\setminus V_0$
spanning the orthogonal complement of $C_0$. Using again Lemma~\ref{cokerC}  it follows that $\rank C_T=1$. Moreover, 
we have from (\ref{eq:dif}) that $V_1$ and $V_0$ are both union of complete relative nullity leaves.  

We claim that the smooth distribution $\Delta\oplus\ker C_{T}$ on $V_1$ is totally geodesic. If $\ker C_{T}$ is locally spanned by a unit 
vector field $X$, then $(\nabla_{X}T)_{\Delta^{\perp}}=0$. From Lemma~\ref{lema autovalor} we have that $\nabla_{T}X=0$. 
Since $\Delta$ is totally geodesic, then  $\Delta\oplus\ker C_{T}$ is integrable.
It remains to show that $\<\nabla_XX,Y\>=0$ where $Y\in\Gamma(\Delta^{\perp})$ is a unit vector field orthogonal to $X$.
Since the only real eigenvalue of $C_T$ is zero, then $C_{T}Y=\mu X$ for a smooth non vanishing function $\mu$.
Lemma~\ref{equation CTXY} yields
$$
(\nabla^h_{X}C_{T})Y=(\nabla^h_{Y}C_{T})X,
$$
which is equivalent to
\begin{equation}\label{equation mu}
X(\mu)=\<\nabla_{Y}Y,X\>\mu 
\end{equation}
and $$
\mu\<\nabla_{X}X,Y\>=0.
$$
The last equation proves the claim.

Since $C_{T}$ is nilpotent, we have that $\ker C_{T}'= \Ima C_{T}'$. From \eqref{A1} we have $C_{T}'AX=0$, and then 
$$
\< AX,X\>=0.
$$
Thus the leaves of $\Delta\oplus \ker C_{T}$ are totally geodesic submanifolds of $\R^{n+1}$, that is, $f|_{V_{1}}$ is ruled.

Next we prove that the rulings contained in $V_1$ are complete. 
Recall that the leaves of relative nullity in $V_1$ are complete.
Assume otherwise that there is an incomplete ruling in $V_1$. Thus, there is a geodesic $\delta\colon [0,a]\to M^n$ in the direction 
of $X$ such that $\delta(a)\notin V_1$. We have 
from Lemma~\ref{regladas 2} that the rank of $f$ at $\delta (a)$ is $2$. Moreover, from the second statement on that lemma, it follows that \eqref{equation mu} extends to $\delta(a)$ where $Y\in\Gamma(\Delta^{\perp})$ is as before. Since $\mu$ is not zero along $\delta$ we have that $\delta(a)\notin V_0$, and hence $\delta(a)\notin V$. On the other hand,
Lemma~\ref{B para reglada} yields that $B$ has the form \eqref{forma de B reglada} with respect to $\{Y,X\}$ and that $\theta\in C^\infty(M)$ verifies \eqref{ecu reglada}. Using again Lemma~\ref{regladas 2}  we see that \eqref{ecu reglada} extends smoothly to $[0,a]$ with $X=\delta'$. But then $B$ has to vanish along $\delta$, and that is a contradiction. 

Let $S$ be a connected component of $V_{1}$ and let $x\in\partial\bar{S}$ together with a sequence $x_{j}\in S$ be such that $x_{j}\to x$.
Let $L_{j}$ be the affine subspace of $\R^{n+1}$ determined by the ruling through $f(x_{j})$. Since the rulings are complete, there is an affine subspace $L$ through $f(x)$ which is the limit of the sequence determined by $L_{j}$. In fact, suppose that there are two subsequences $L'_j$ and $L''_j$ converging to different subspaces $L'$ and $L''$ that intersect at $f(x)$. 
Then, in a neighborhood of $x$ different subspaces $L'_j$ and $L''_j$  would  intersect, and this is a contradiction.
Clearly $L\subset f(\partial\bar{S})$, and thus $f|_{\bar{S}}$ is a ruled strip.

Notice that if two ruled strips have common boundary then their union is also a ruled strip.
Take $x\in V_0$. Since $V_1$ is dense in $V$, then $f(x)\in L\subset f(M)$ where $L$ is an affine $(n-1)$-dimensional
subspace of $\R^{n+1}$ that is the limit of a sequence of rulings. Suppose that there exist two sequences of rulings $L'_j$ and $L''_j$ converging to affine subspaces $L'\neq L''$ that intersect at $f(x)$. Then $L'_j$  intersects $L''$ in a hyperplane for large values of $j$. Fixing $j$ large enough, the same holds for any ruling in a neighborhood of rulings of $L'_j$. 

Let $Z'$ and $Z''$ be vector fields tangent to $L'_j$ and $L''$, respectively, and let $R$ be a vector field tangent to 
$L''\cap L'_j$. Since $\tilde{\nabla}_R Z'$ and $\tilde{\nabla}_R Z''$ have no normal components,
it follows that $L''\cap L'_j$ is a complete relative nullity leaf. The same holds for the nearby rulings.
In a neighborhood of $y\in L''\cap L'_j$, as before take unit vector fields $T\in\Gamma(C_0^\perp)$, $X\in\Gamma(\ker C_T)$
and $Y$ such that $C_T Y=\mu X$ with $\mu\neq 0$.
Let $\gamma$ be the unit speed geodesic of $M^n$ such that $f\circ \gamma$ lies in $L''$, $f(\gamma(0))=y$ and is orthogonal to $\Delta$. Then $\gamma'=aX+bY$ with $b\neq0$.
Hence $\<C_T\gamma',\gamma'\>=0$ is equivalent to $ab\mu=0$.
This yields $a=0$, and thus $\gamma'=Y$ is orthogonal to $X$. Since $f_*\nabla_{\gamma'} T$ is tangent to $L''$ and $f_*X$ is orthogonal to $L''$, then $C_T Y=0$, and this is a contradiction. Therefore, we have seen that any sequence of points in $V_1$ converging to $x$, determines the same affine subspace $L$ as the limit of the correspondent rulings. Moreover, we have shown that $L$ does not intersect $f(V_1)$.

We have proved that there exists an open neighborhood $U$ of $x$ such that $f|_U$ is ruled and has complete relative nullity leaves. 
Using Lemma~\ref{B para reglada} as above, we obtain that the affine subspace $L$ is contained in $f(V_0)$. Hence, every connected component of $V$ defines a ruled strip. 

To conclude the proof of the theorem it remains to show that 
$B=0$ on the open subset $W_1=M_1\setminus \bar{M_2}$, that is, that $B$ vanishes outside ruled strips.
It follows from \eqref{B1} that $B(\Delta)\subset \Ima A$, hence $\rank B\leq 2$ on $W_{1}$.
Let $V'$ be the open subset of $W_{1}$ defined as
$$
V'=\left\{x\in W_{1}: \rank B(x)=2\right\}.
$$
Take local orthonormal vector fields $X$ and $Y$ in $V'$ orthogonal to $\ker B$ such that $X$ is an eigenfield of $A$. 
Then $A$ and $B$ have the expressions 
$$
A|_{\ker B^{\perp}}=\begin{bmatrix}
\lambda & 0\\ 0 & 0
\end{bmatrix},\;\; B|_{\ker B^{\perp}}=\begin{bmatrix}
\mu  & b \\b & 0
\end{bmatrix}
$$
with respect to the frame $\{X,Y\}$ and $\lambda\neq 0\neq b$.

We claim that $V'$ is empty.  Suppose otherwise.
Given $T\in\Gamma(\Delta)$ let $c_{T}$  be defined by $C_{T}X=c_{T}X$. 
Since $X$ is parallel along the relative nullity leaves, we have
$$
(\nabla_{X}B)Y
=(X(b)+c_Y\mu)X+2c_Yb Y+b(\nabla_{X}X)_{\ker B}
$$
and 
$$
(\nabla_{Y}B)X=Y(\mu)X+Y(b)Y+b\nabla_{Y}Y.
$$
In particular \eqref{B2} yields
\begin{equation}\label{codB1} 
Y(b)=2c_{Y}b.
\end{equation}
Similarly, we obtain 
\begin{equation}
\begin{cases}\label{codB2}
S(\mu)=c_S\mu -b\<\nabla_{X}S,Y\>\\
S(b)=c_{S}b\\ 
\nabla_{S}Y=0
\end{cases}
\end{equation} 
for any $S\in\Gamma(\ker B)$ on $V'$.  

By the last equation in \eqref{codB2}, we have that $\ker B$ is a totally geodesic distribution 
along $V'$. We see next that its leaves are complete. On the contrary, suppose that there is a geodesic $\gamma\colon [0,a]\to M^n$ 
such that $\gamma([0,a))\subset V'$ is contained  on a leaf of $\ker B$ and that $\gamma(a)\notin V'$. 
By Lemma~\ref{lema eq} the relative nullity subspace at $\gamma(a)$ has dimension $n-1$. Then, we have one of the following possibilities:
\begin{itemize}
\item[(i)] $\gamma(a)\in W_1$ and $\rank B(\gamma(a))\leq 1$,
\item[(ii)] $\gamma(a)\in \bar{M}_2$ and $B(\gamma(a))=0$, 
\item[(iii)] $\gamma(a)\in \bar{M}_2$ and $B(\gamma(a))\neq 0$.
\end{itemize}

We first show that the latter possibility cannot occur. Assume that $(iii)$ holds and take a neighborhood of $\gamma(a)$ where $B\neq 0$. Since $\gamma(a)\in\bar{M}_2$, there is a sequence $x_k\in V$ such that $x_k\to \gamma(a)$. Recall that each connected component of $V$ defines a ruled strip. Let $L_k$ be the affine subspace of $\R^{n+1}$ given by the ruling through $f(x_k)$. As before, there is an affine subspace $L\subset f( \bar{M_2})$ of dimension $n-1$ which is the limit of the sequence $L_k$ and contains $f(\gamma(a))$. Since $A\gamma'(a)=0$ and the geodesic $f\circ\gamma$ is transversal to $L$, we have that $A(\gamma(a))=0$, and that is a contradiction.

In the two remaining cases we have that $b(\gamma(a))=0$. The second equation in \eqref{codB2} extends to $[0,a]$ implying that $b$ vanishes along $\gamma$, and this leads to a contradiction. Hence $\ker B$ has complete leaves in $V'$.

The leaves of the relative nullity foliation cannot be complete on any open subset of $W_{1}$. This follows easily from Lemma~\ref{lema autovalor} and the assumptions on $f$.
Let $W_{1}'\subset W_{1}$ be the dense subset where the the relative nullity leaves are not complete. 
Take a point $x\in V'\cap W'_1$. Since the leaf of the relative nullity foliation through $x$ is not complete, there is a geodesic $\delta\colon [0,a]\to M^n$ contained in that leaf tangent to $Y$ such that $\delta([0,a))\subset V'$ and $\delta(a)\notin V'$. By the same transversality argument as 
above we see that $b(\delta(a))=0$. It
follows from \eqref{codB1} that $b=0$ along $\delta$, and that is a contradiction proving the claim that $V'$ is empty.

We have that $\rank B\leq 1$ on $W_1$. Given $x\in W'_1$, there is a geodesic $\gamma\colon [0,a]\to M^n$ with 
$\gamma([0,a))\subset W_{1}'$ contained in a leaf of the relative nullity such that $\gamma(0)=x$ and $B(\gamma(a))=0$. 
Now the first equation in \eqref{codB2} (for $b=0$) gives that $B$ vanishes along $\gamma$. Thus $B$ vanishes on $W'_1$, and
hence on $W_1$. This concludes the proof of the theorem.

\begin{corollary}\label{corollary}\po
Let $f\colon M^n\to \R^{n+1}$ be an isometric immersion of a simply connected Riemannian manifold $M^n$ satisfying the hypothesis of Theorem~\ref{main}. If $\tau$ is a 
non-trivial infinitesimal bending of $f$, then $\tau$ is the variational field of an isometric bending.
\end{corollary}
\begin{proof}
Let $B$ be the symmetric tensor associated to the infinitesimal bending $\tau$. It is easy to see using \eqref{forma de B reglada} and \eqref{ecu reglada} that the symmetric tensors $A+tB$, $t\in\R$,  satisfy the Gauss and Codazzi equations. Then, they give rise to an isometric bending of $f$ having $\tau$ as its variational field. \qed
\end{proof}

\Address

\end{document}